\documentclass[12pt, reqno]{amsart}

\usepackage{amsmath, amsthm, amscd, amsfonts, amssymb, graphicx, color}
\usepackage[bookmarksnumbered, colorlinks, plainpages]{hyperref}
\input{mathrsfs.sty}
\hypersetup{colorlinks=true,linkcolor=red, anchorcolor=green, citecolor=cyan, urlcolor=red, filecolor=magenta, pdftoolbar=true}

\newtheorem{theorem}{Theorem}[section]
\newtheorem{lemma}[theorem]{Lemma}
\newtheorem{proposition}[theorem]{Proposition}
\newtheorem{corollary}[theorem]{Corollary}
\theoremstyle{definition}
\newtheorem{definition}[theorem]{Definition}

\theoremstyle{remark}
\newtheorem{remark}[theorem]{Remark}
\numberwithin{equation}{section}

\newcommand{\med}{{\rm med}}
\begin{document}

\title[Maximal inequalities in quantum probability spaces]{Maximal inequalities in quantum probability spaces}

\author[Gh. Sadeghi, A. Talebi, M.S. Moslehian, ]{Gh. Sadeghi $^{1,3}$, A. Talebi$^{2,3}$ \MakeLowercase{and} M. S. Moslehian$^{2,3}$}

\address{$^1$ Department of Mathematics and Computer Sciences, Hakim Sabzevari University, P.O. Box 397, Sabzevar, Iran}
\email{g.sadeghi@hsu.ac.ir}

\address{$^2$ Department of Pure Mathematics, Ferdowsi University of Mashhad, P.O. Box 1159, Mashhad 91775, Iran.}
\email{talebi.ali@mail.um.ac.ir}
\email{moslehian@um.ac.ir; moslehian@yahoo.com}

\address{$^3$ Center Of Excellence in Analysis on Algebraic Structures (CEAAS), Ferdowsi University of Mashhad, Iran}


\begin{abstract}
We employ some techniques involving projections in a von Neumann algebra to establish some maximal inequalities such as the strong and weak symmetrization, L\'{e}vy, L\'{e}vy--Skorohod, and Ottaviani inequalities in the realm of quantum probability spaces. 
\end{abstract}

\maketitle

\section{Introduction and preliminaries}

Some of the important kinds of probability inequalities such as Kolmogorov, L\'{e}vy, and Ottaviani inequalities relate tail probabilities for the maximal partial sum of independent random variables; see \cite{LR, PIN, SZE}. By centering sums of independent random variables at corresponding medians, Paul L\'{e}vy \cite{L} obtained some maximal inequalities, which can play a similar role as Kolmogorov's inequalities; see also \cite{SUK}. In fact, he proved the probability of a maximal partial sum of random variables exceeds some given number is concerned with the probability that the last partial sum does so.

L\'{e}vy inequalities (see also \cite{SAK}) assert that if $X_1, X_2, \ldots, X_n$ are symmetric independent random variables with partial sums $S_k, k=1, 2, \ldots, n$, then
for any $\lambda > 0$
\begin{eqnarray*}
&&\mathbb{P}\left( \max_{1\leq k \leq n} \left(S_k \right) > \lambda\right) \leq 2 \mathbb{P}(S_n > \lambda),\\
\text{and}\\
&&\mathbb{P}\left( \max_{1\leq k \leq n} \left| S_k \right| > \lambda\right) \leq 2 \mathbb{P}(|S_n| > \lambda).
\end{eqnarray*}
Many classical inequalities have been extended to the quantum setting. We refer to \cite{BEK, TMS1, TMS2} and the references therein for more information.

In what follows, we give some necessary preliminaries on quantum probability spaces. Throughout this paper, we denote by $\mathfrak{M}$ a von Neumann algebra on a Hilbert space $\mathcal{H}$ with the unit element $\textbf{1}$ equipped with a normal faithful tracial state $\tau$. The elements of $\mathfrak{M}$ are called (noncommutative) random variables. We denote by $\leq$ the usual order on the self-adjoint part $\mathfrak{M}^{sa}$ of $\mathfrak{M}$. For any projection $q \in \mathfrak{M}$, we write $q^{\perp}=\textbf{1}-q$.\\
For every self-adjoint operator $x\in \mathfrak{M}$, there exists a unique spectral measure $E$ supported on the spectrum $\sigma(x)$ of $x$ such that $x = \int_{\mathbb{R}} \lambda dE$. Recall that the resolution of identity $\left(e_{\lambda}(x) \right)_{\lambda}$ for $x$ is an increasing family of projections such that for any $\lambda \in \mathbb{R}$, $e_{\lambda}(x) = E((-\infty, \lambda))$ is the spectral projection of $x$ corresponding to the interval $(-\infty, \lambda)$. Moreover, we denote by $e_B(x)$ the spectral projection $E(B)$ of $x$ corresponding to the Borel subset $B$.
If $p \geq 1$, then the noncommutative $L_p$-space $L_p(\mathfrak{M})$ is defined as the completion of $\mathfrak{M}$ with respect to the $L_p$-norm $\| x \|_p := \left(\tau \left( |x|^p \right)\right)^{\frac{1}{p}}$. Further, if $x \in \mathfrak{M}$ is a positive element, then 
\[\| x \|_p^p = \int_0^{\infty} p t^{p-1} \tau \left( e_t^\perp (x) \right) dt.\]
Further, if $x \in L_p\left(\mathfrak{M}\right)$ is self-adjoint and $t>0$, then we have the inequality
\begin{eqnarray}\label{CI}
\tau(e_{t}^{\perp}(x))\leq t^{-p}\tau(|x|^p),
\end{eqnarray}
which is known as the Chebyshev inequality in the literature (see \cite{Lu}). \\
We need the next result related to the lattice of projections.
\begin{lemma}[\cite{N}]\label{PR}
Let $p$ and $q$ be two projections of $\mathfrak{M}$. Then\\
(i) if $(p_{\lambda})_{\lambda\in\Lambda}$ is a family of projections in $\mathfrak{M}$, then $\tau\left(\vee_{\lambda\in\Lambda}p_{\lambda}\right)\leq\sum_{\lambda\in\Lambda}\tau(p_{\lambda})$.\\
(ii) if $p$ and $q$ commute, then $p\wedge q=pq$.
\end{lemma}


Some tools in the study of classical results still work in the extension to noncommutative setup. However, some techniques must be invented to find quantum versions of classical inequalities involving maximum of random variables. Indeed, the maximum of two self-adjoint operators does not exist, in general. To do away with this problem, we shall construct some special projections to establish some noncommutative maximal inequalities.

In this paper, we intend to prove analogues of several classical maximal inequalities such as strong symmetrization, L\'{e}vy, L\'{e}vy--Skorohod, and Ottaviani inequalities in the noncommutative setup. It is noteworthy that, in quantum probability theory, various concepts of independence have been studied such as tensor independence, free independence (freeness), and Boolean independence. We use the notion of the tensor independence and a weaker notion of full independence, say weakly full independence, to obtain noncommutative counterpart of some maximal inequalities. 

\section{Noncommutative L\'{e}vy inequality}

We start our work with the noncommutative counterpart of the notion of median.

\begin{definition}
For a self-adjoint element $x \in \mathfrak{M}$, we say a real number $m$ is the \textit{median} of $x$ if the following inequalities hold:
\begin{align*}
\tau\left(e_{(-\infty, m]}(x)\right) \geq \frac{1}{2} \quad \text{and} \quad \tau\left(e_m^{\perp}(x)\right) \geq \frac{1}{2}.
\end{align*}
\end{definition}
The median of $x$ is denoted by $\med(x)$. The median of any self-adjoint element $x$ always exists. In fact, the real number $m := \sup\{ \alpha: \tau\left(e_{\alpha}^{\perp}(x)\right) \geq \frac{1}{2}\}$ is a median of $x$. To observe this, let $\Sigma:= \{ \alpha: \tau\left(e_{\alpha}^{\perp}(x)\right) \geq \frac{1}{2}\}$, which is nonempty, since $-\|x\| \in \Sigma$.
Furthermore, $\Sigma$ is bounded above, because if $\alpha > \| x \|$, then $e_{\alpha}^{\perp}(x) = 0$, and hence $\alpha \notin \Sigma$.
 Consider a decreasing sequence $(\alpha_n)$ of real numbers converging to $m$. Then $(-\infty, m] = \bigcap\limits_{n=1}^{\infty} (-\infty, \alpha_n)$.
On the other hand, there exists an increasing sequence $(\beta_n)$ in $\Sigma$ such that $\beta_n \nearrow m$. Hence, $ [m, \infty) = \bigcap\limits_{n=1}^{\infty}[\beta_n, \infty)$.
Assume that $e(x)$ is the spectral measure for $x$ and let $\xi \in \mathcal{H}$. Since the function $e^{\xi, \xi}(x): B \mapsto \langle e_B(x)\xi, \xi \rangle$ is a regular Borel measure, we have
\begin{eqnarray*}
&& e^{\xi, \xi}((-\infty, \alpha_n)) \longrightarrow e^{\xi, \xi}((-\infty, m]) \qquad \text{and} \qquad 
e^{\xi, \xi}([\beta_n, \infty)) \longrightarrow e^{\xi, \xi}([m, \infty)) 
\end{eqnarray*}
That is
\begin{eqnarray*}
&& \langle e_{\alpha_n}(x)\xi, \xi \rangle \longrightarrow \langle e_{(-\infty, m]}(x)\xi, \xi \rangle \quad \text{and} \quad \langle e_{\beta_n}^{\perp}(x) \xi, \xi \rangle \longrightarrow \langle e_m^{\perp}(x)\xi, \xi \rangle 
\end{eqnarray*}
and hence in the strong operator topology
\begin{eqnarray*}
e_{\alpha_n}(x) \longrightarrow e_{(-\infty, m]}(x) \quad \text{and} \quad
e_{\beta_n}^{\perp}(x) \longrightarrow e_m^{\perp}(x)
\end{eqnarray*}
Thus
\begin{eqnarray*}
\tau(e_{(-\infty, m]}(x)) = \lim_{n \rightarrow \infty}\tau\left( e_{\alpha_n}(x) \right) =
1 - \lim_{n \rightarrow \infty}\tau\left( e_{\alpha_n}^{\perp}(x) \right) \geq \frac{1}{2}
\end{eqnarray*}
and
\begin{eqnarray*}
\tau\left(e_m^{\perp}(x)\right) = \lim_{n \rightarrow \infty}\tau\left(e_{\beta_n}^{\perp}(x)\right) \geq \frac{1}{2}.
\end{eqnarray*}

Some properties of the median are presented in the next proposition.
\begin{proposition}
Let $x \in \mathfrak{M}^{sa}$, $p \geq 1$, and $\alpha$ be a positive real number, then
\begin{itemize}
\item[\textbf{(i)}] if $\tau\left(e_{\alpha}^{\perp}(|x|)\right) < \frac{1}{2}$, then $|\med(x)| \leq \alpha$.\\
\item[\textbf{(ii)}] $| \med(x) | \leq 2^{\frac{1}{p}}\| x \|_p$.\\
\item[\textbf{(iii)}] $|\med(x) - \tau(x)| \leq \sqrt{2\, {\rm var}(x)}$, where ${\rm var}(x) = \tau\left(x^2\right) - \tau(x)^2$.
\end{itemize}
\end{proposition}
\begin{proof}
\textbf{(i)} From $e_{\alpha}^{\perp}(|x|) = e_{\alpha}^{\perp}(x) + e_{\alpha}^{\perp}(-x)$ and the assumption we have $\tau(e_{\alpha}^{\perp}(x)) < \frac{1}{2}$ and $\tau(e_{\alpha}^{\perp}(-x)) < \frac{1}{2}$. If $\med(x) > \alpha$, then we get $\tau(e_m^{\perp}(x)) \leq \tau(e_{\alpha}^{\perp}(x)) < \frac{1}{2}$, which is impossible; hence $\med(x) \leq \alpha$. Similarly, $\med(x) \geq -\alpha$.\\
\textbf{(ii)} The conclusion can be deduced from (i) and the Chebyshev inequality \eqref{CI}
\begin{align*}
\tau\left(e_{2^{\frac{1}{p}}\| x \|_p}^{\perp}(|x|)\right) \leq \frac{\|x\|_p^p}{(2^{\frac{1}{p}}\| x \|_p)^p} = \frac{1}{2}.
\end{align*}
\textbf{(iii)} It is enough to use \textbf{(ii)} with $x - \tau(x)$ instead of $x$, and the fact that $||\, x - \tau(x)\, \|_2^2 = {\rm var}(x)$.
\end{proof}

In quantum probability theory, there is no single notion of independence; cf. \cite{Fr1, Mur}. Among noncommutative meanings of independence, freeness \cite{VDN}, which was introduced by Voiculescu and led to the development of free probability theory; tensor independence \cite{D, JX2} which is a straightforward generalization of the notion in classical probability theory and gives us a way to compute mixed moments from the moments of the summands; Boolean independence \cite{Liu}, which is related to full free product of algebras \cite{CF}; and monotone independence \cite{Mur}, which have been considered as the most fundamental one. Franz \cite{Fr3} studied some relations between freeness, monotone independence and boolean independence via B\.ozejko and Speicher’s two-state free products \cite{BS}.

To establish our main result, we can assume an independence condition as defined in \cite{TMS1}, which is weaker than the tensor independence. A sequence $(x_k)_{k=1}^n$ is said to be weakly fully independent if the subalgebras $W^*(x_1, \ldots, x_{j-1})$ and $W^*(x_j, \ldots, x_n)$ are independent for any $1 < j \leq n$, in the sense that
\begin{equation*}
\tau(ab) = \tau(a)\tau(b), \qquad a \in W^*(x_1, \ldots, x_{j-1}), b \in W^*(x_j, \ldots, x_n)
\end{equation*}
in which $W^*(S)$ denotes the $W^*$-algebra generated by a subset $S$ of $\mathfrak{M}$.

The next theorem provides a noncommutative analogue of the classical maximal L\'{e}vy inequality. 
Recall that two normal random variables $x, y$ are identically distributed if $\tau \left( e_{B}(x) \right) = \tau \left( e_{B}(y) \right)$ for any complex Borel set $B$. 
A random variable $x \in \mathfrak{M}^{sa}$ is said to be symmetric if $x$ and $-x$ are identically distributed. 

The following Lemma shows that any finite sum of weakly fully independent symmetric random variables in $\mathfrak{M}^{sa}$ is symmetric.


\begin{lemma}\label{Lem1}
If $x, y \in \mathfrak{M}^{sa}$ are two weakly fully independent symmetric random variables, then $x + y$ is symmetric.
\end{lemma}
\begin{proof}
By the binomial identity and the tracial property of $\tau$, we get for every natural number $k$
\begin{eqnarray*}
\tau \left( (x+ y)^k \right) & = & \sum_{i= 0}^k \binom{k}{i} \tau \left( x^i\, y^{k-i} \right) \\
& = & \sum_{i= 0}^k \binom{k}{i} \tau \left( x^i \right) \tau \left( y^{k-i} \right) \quad\qquad\quad ~~ (\text{by weakly fully independence}) \\
& = & \sum_{i= 0}^k \binom{k}{i} \tau \left( (-x)^i \right) \tau \left( (-y)^{k-i} \right) \quad (\text{by symmetricity of $x, y$}) \\
& = & \sum_{i= 0}^k \binom{k}{i} \tau \left( (-x)^i\, (-y)^{k-i} \right) \quad\quad ~~ (\text{by weakly fully independence}) \\
& = & \tau \left( (-x - y)^k \right),
\end{eqnarray*}
and hence the probability distribution of $x + y$ and $-x - y$ coincide by \cite[Page 203, Remark]{D}.
\end{proof}
\begin{theorem}[Noncommutative L\'{e}vy inequality]\label{th1}
Let $x_1, x_2, \ldots x_n$ be weakly fully independent symmetric random variables in $\mathfrak{M}^{sa}$ with the partial sums $s_k$ satisfying $s_ks_n = s_ns_k$ for all $1\leq k\leq n$. Then, for each  $\lambda>0$, there exist two projections $p$ and $q$ such that
\begin{eqnarray}
&&  \max_{1 \leq k \leq n} \frac{1}{2^{k-1}} \tau\left( e_{(\lambda, \infty)}\left(s_k \right) \right) ~ \leq ~ \tau(p) \leq ~  2 \tau \left( e_{(\lambda, \infty)}\left( s_n \right) \right),\label{in1-1}
\end{eqnarray}
and
\begin{eqnarray}
\max_{1 \leq k \leq n} \frac{1}{2^{k-1}} \tau\left( e_{(\lambda, \infty)}\left( |s_k| \right) \right) ~ \leq ~ \tau(p + q) \leq 2 \tau \left( e_{(\lambda, \infty)}\left( | s_n | \right) \right). \label{in1-2}
\end{eqnarray}
Then, if $e_{(\lambda, \infty)}\left(s_k \right)$ (respectively, $e_{(\lambda, \infty)}\left(-s_k \right)$) is nonzero for some $1 \leq k \leq n$, then $p$ (respectively $q$) is a nonzero projection. 
\end{theorem}
\begin{proof}
The proof is based on some constructions of special sequence of projections relative to $(s_k)$ and a parameter $\lambda$. More precisely, consider projections $r_k = e_{(-\infty, \lambda]}\left( s_k \right)$ with $1 \leq k \leq n$. We set
\begin{eqnarray}
&& p_1 := r_1^{\perp}, \quad p_k := \bigwedge_{i=1}^{k-1} r_{i} \wedge r_k^{\perp} \in W^*\left(x_1, x_2, \ldots, x_k \right) \quad (1 < k \leq n) \label{in5-1};\nonumber\\
&&t_k := e_{[0, \infty)}\left( s_n - s_k \right) \label{in5-2};\\
&&f_k := r_k^{\perp} \wedge t_k = r_k^{\perp} t_k \label{in5-3}\nonumber,
\end{eqnarray}
for each $1 \leq k \leq n$.
Actually, $ (p_k)_k$ is a sequence of orthogonal projections. 
Since $s_k$ and $s_n$ commute, considering the abelian von Neumann algebra $W^*(s_k, s_n)$ generated by $s_k$ and $s_n$, one can deduce that
\begin{eqnarray}\label{in7}
f_k = r_k^{\perp} t_k  \leq e_{(\lambda, \infty)}\left( s_n \right)
\end{eqnarray}
for all $k$.
Note that $p_kf_k = p_kt_k$, since $p_k \leq r_k^{\perp}$. From \eqref{in7}, multiplying by $p_k$, we have
\begin{eqnarray}\label{in6}
p_kt_kp_k & = & p_k r_k^{\perp}t_k p_k \leq p_k\, e_{(\lambda, \infty)}\left( s_n \right)\, p_k \nonumber \\
& \Longrightarrow & \tau \left( p_k t_k \right) \leq \tau \left( p_k e_{(\lambda, \infty)}\left( s_n \right) \right) \nonumber \\
& \Longrightarrow &
\sum_{k=1}^n \tau \left( p_k t_k \right) \leq \sum_{k=1}^n \tau \left( p_k\, e_{(\lambda, \infty)}\left( s_n \right) \right)
\leq \tau \left( e_{(\lambda, \infty)}\left( s_n \right) \right).
\end{eqnarray}
Putting $p := \sum_{k=1}^n p_k$, the right hand side of inequality \eqref{in1-1} is deduced from
\begin{eqnarray}\label{in2}
\tau \left( e_{(\lambda, \infty)}\left( s_n \right) \right) &\geq & \sum_{k=1}^n \tau(p_k\, t_k) \quad \qquad (\text{by \eqref{in6}}) \nonumber \\
&=& \sum_{k=1}^n \tau(p_k)\tau(t_k) \quad \left(\text{by the weakly full independence}\right) \nonumber\\
&\geq & \frac{1}{2}\tau(p)
\end{eqnarray}
in which the last inequality can be obtained from \eqref{in5-2}, Lemma \ref{Lem1} which implies that $\med\left( s_k - s_n \right) = 0$, and the definition of the median.

By utilizing the same argument with $-s_k$ instead of $s_k$, we get a projection $q$ such that
\begin{eqnarray}\label{in3}
\tau \left( e_{(\lambda, \infty)}\left( -s_n \right) \right) \geq \frac{1}{2}\tau(q).
\end{eqnarray}
Now, the right hand side of inequality \eqref{in1-2} can be obtained by summing \eqref{in2} and \eqref{in3}.

Next, in order to show the left hand side of inequlity \eqref{in1-1}, we prove by strong induction that 
\begin{equation}\label{ine3}
\tau(p^{\perp}) \leq 1 - \frac{1}{2^{k-1}}\tau \left(r_k^{\perp}\right)
\end{equation}
for every $1 \leq k \leq n$. The step $k = 1$ holds due to $\tau\left(p^{\perp} \right) \leq \tau(p_1^{\perp}) = 1 - \tau(r_1^{\perp})$.
Now, let inequality \eqref{ine3} holds for any $i <k $. Then
\begin{eqnarray*}
\tau(p^{\perp}) & \leq & \tau \left(p_k^{\perp} \right) = \tau \left( \bigvee_{i=1}^{k-1} r_{i}^{\perp} \vee r_k \right) \\
& \leq & \sum_{i=1}^{k-1} \tau \left( r_{i}^{\perp} \right) + \tau \left( r_{k} \right) \qquad (\text{by Lemma \ref{PR} \textbf{(i)}}) \\
& \leq & \left( 1 - \tau \left( p^{\perp} \right) \right) \sum_{i=1}^{k-1} 2^{i-1} + \tau \left( r_{k} \right) \quad \left(\text{by \eqref{ine3} and the inductive hypothesis}\right)\\
& = & \left( 1 - \tau \left( p^{\perp} \right) \right) \left( 2^{k-1} - 1\right) + \tau \left( r_{k} \right) 
\end{eqnarray*}
Therfore,
\begin{eqnarray*}
\tau(p^{\perp}) \leq \frac{2^{k-1} - 1 + \tau \left( r_{k} \right) }{2^{k-1}} = 1 - \frac{1}{2^{k-1}}\tau \left(r_k^{\perp}\right),
\end{eqnarray*}
and this completes the argument.\\
It follows from inequality \eqref{ine3} that $\tau (p) \geq \frac{1}{2^{k-1}} \tau \left( r_k^{\perp} \right)$ for every $1 \leq k \leq n$, and so we arrive at \eqref{in1-1}.
It can be proved by the same reasoning that $\max\limits_{1 \leq k \leq n} \frac{1}{2^{k-1}} \tau\left( e_{(\lambda, \infty)}\left( -s_k \right) \right) \leq \tau(q)$. Hence,
\begin{eqnarray*}
&&  \max_{1 \leq k \leq n} \frac{1}{2^{k-1}} \tau\left( e_{(\lambda, \infty)}\left( |s_k| \right) \right) \\
&=& \max\limits_{1 \leq k \leq n} \frac{1}{2^{k-1}} \tau\left( e_{(\lambda, \infty)}\left( s_k \right) \right) + 
\max\limits_{1 \leq k \leq n} \frac{1}{2^{k-1}} \tau\left( e_{(\lambda, \infty)}\left( -s_k \right) \right) \leq  \tau(p + q),
\end{eqnarray*}
and the proof of the left hand side of \eqref{in1-2} is complete.
\end{proof}

\begin{remark}
Notice that, for example, if $\mathfrak{M}_1, \ldots, \mathfrak{M}_n$ are some noncommutative probability spaces and $x_i \in \mathfrak{M}_i$, then $x_1, x_2, \ldots, x_n$ satisfy the commutative condition $s_ks_n = s_ns_k$ for all $k$ stated in Theorem \ref{th1} with respect to the product probability space $\bigotimes\limits_{i=1}^n \mathfrak{M}_i$, by identifying $x_i$ with $1 \otimes \ldots \otimes 1 \otimes x_i \otimes 1 \ldots \otimes 1$ ($x_i$ on the $i$-th spot). 
Furthermore, an example provided in \cite{TMS1} to show that our condition is weaker than commutativity of all $x_i$'s.
However, we present another example in the matrix algebra $\mathbb{M}_3(\mathbb{C})$ for the reader's convenience.
\begin{eqnarray*}
&& x_1 = \begin{pmatrix}
1 & 1-i & 0 \\
1+i & 3 & i \\
0 & -i & -1
\end{pmatrix},\qquad \qquad
x_2 = \begin{pmatrix}
0 & -1 & -i \\
-1 & 1 & 2i \\
i & -2i & 3 
\end{pmatrix}\\
&& x_3 = \begin{pmatrix}
3 & 2i & i+1 \\
-2i & -2 & 1 \\
1-i & 1 & 2
\end{pmatrix},
\quad \text{and}\qquad
x_4 = \begin{pmatrix}
-2 & -i & -1 \\
i & 0 & -3i-1 \\
-1 & -1+3i & -2
\end{pmatrix}.
\end{eqnarray*}
Then, not all $x_1, x_2, x_3, x_4$ commute with each other, but $s_1$, $s_2$ and $s_3$ commute with $s_4 = \begin{pmatrix}
2 & 0 & 0 \\
0 & 2 & 0 \\
0 & 0 & 2
\end{pmatrix}$
\end{remark}

Now, we present the classical version of L\'{e}vy inequality.
\begin{corollary}
Let $X_1, X_2, \ldots, X_n$ be symmetric independent random variables in probability space $(\Omega, \mathfrak{F}, \mathbb{P})$ with partial sums $S_n$. Then, for any $\lambda > 0$
\begin{eqnarray}
&&\mathbb{P}\left( \max\limits_{1\leq k \leq n} \left(S_k \right) > \lambda\right) \leq 2 \mathbb{P}(S_n > \lambda), \label{cor1}\\
\text{and}~
&&\mathbb{P}\left( \max\limits_{1\leq k \leq n} \left| S_k \right| > \lambda\right) \leq 2 \mathbb{P}(|S_n| > \lambda)\label{cor2}
\end{eqnarray}
\end{corollary}
\begin{proof}
The projections $r_k, p_k$ in the proof of noncommutative L\'{e}vy inequality correspond to the characteristic functions of the subsets $R_k = R_{k-1} \cap \{S_k \leq \lambda\}$ and $P_k = R_{k-1}\cap \{S_k > \lambda\}$ with $R_0=\Omega$.
Note that the projections $p$ and $q$ in Theorem \ref{th1} correspond to the characteristic functions of the subsets $\{\max\limits_{1\leq k \leq n} (S_k) > \lambda\}$ and $\{\max\limits_{1\leq k \leq n} (-S_k) > \lambda\}$, respectively. Indeed, it is easy to check that
\begin{eqnarray*}
\bigcup_{k=1}^n P_k = \{\max_{1\leq k \leq n} (S_k) > \lambda\}.
\end{eqnarray*}
Thus, by Theorem \ref{th1}
\begin{eqnarray*}
\mathbb{P}\left( \max_{1\leq k \leq n} \left(S_k \right) > \lambda\right)  \leq 2\mathbb{P}(S_n > \lambda),
\end{eqnarray*}
which ensures inequality \eqref{cor1}. Inequality \eqref{cor2} follows from
\begin{eqnarray*}
\{\max_{1\leq k \leq n} (S_k) > \lambda\} \cup \{\max_{1\leq k \leq n} (-S_k) > \lambda\} = \{\max_{1\leq k \leq n} |S_k| > \lambda\}.
\end{eqnarray*}
\end{proof}

In what follows, we employ the same reasoning as in the proof of noncommutative L\'{e}vy inequality, so we omit some details.\\
Ottaviani and L\'{e}vy--Skorohod inequalities are related to the tail probabilities for maximum partial sums of some random variables which are not symmetric.
The first result of this section reads as follows.  

\begin{theorem}[Noncommutative Ottaviani inequality]\label{O}
Let $x_1, x_2, \ldots, x_n$ be weakly fully independent self-adjoint random variables in $\mathfrak{M}$ such that $s_ks_n = s_ns_k$ $(k=1, 2, \ldots, n)$, where $s_k$ are the partial sums. Then for each $\lambda>0$, setting $M_{\lambda} := \frac{1}{\min\limits_{1 \leq k\leq n} \tau\left(e_{(-\infty, \frac{\lambda}{2}]}\left(\,|s_n - s_k|\, \right) \right)}$ (with $1/0=\infty$), there exists a projection $p$ such that
\begin{eqnarray*}
\max\limits_{1 \leq k \leq n} \frac{1}{2^{k-1}} \tau \left( e_{(2\lambda, \infty)}\left(|s_k| \right) \right) ~ \leq ~ \tau(p) ~ \leq ~ M_{\lambda}\, \tau \left( e_{(\frac{\lambda}{2}, \infty)}\left( |s_n| \right) \right)
\end{eqnarray*}
\end{theorem}
\begin{proof}
Consider the sequence $r_k = e_{(-\infty, 2\lambda]}\left( |s_k| \right) \,\, 1 \leq k \leq n$ of projections with respect to $|s_k|$ and the parameter $2\lambda$. We set
\begin{eqnarray*}
&&p_k := \bigwedge_{i=1}^{k-1} r_{i} \wedge r_k^{\perp} \\
\text{and} \quad
&&t_k := e_{(-\infty, \frac{\lambda}{2}]}\left( |s_n - s_k| \right)
\end{eqnarray*}
for each $1 \leq k \leq n$.
Then $ (p_k)_k$ is a sequence of orthogonal projections. Setting $f_k := t_k \wedge r_k^{\perp} = t_k r_k^{\perp}$, we get
\begin{eqnarray*}\label{Ot1}
f_k \leq e_{(\frac{\lambda}{2}, \infty)}\left( |s_n| \right)
\end{eqnarray*}
and hence $p_k\, t_k\, p_k \leq p_k\, e_{(\frac{\lambda}{2}, \infty)}\left( |s_n| \right) p_k$.
Consequently, taking the trace and summing over $k$, we have
\begin{eqnarray}\label{Ot2}
\sum_{k=1}^n \tau(p_k\, t_k) \leq \sum_{k=1}^n \tau\left( p_k\, e_{(\frac{\lambda}{2}, \infty)}\left( |s_n| \right) \right) \leq \tau\left( e_{(\frac{\lambda}{2}, \infty)}\left( |s_n| \right) \right)
\end{eqnarray}
Putting $p := \sum_{k=1}^n p_k$, the desired inequality follows via
\begin{eqnarray*}
\tau \left( e_{(\frac{\lambda}{2}, \infty)}\left( |s_n| \right) \right) &\geq & \sum_{k=1}^n \tau(p_k\, t_k) \quad \qquad (\text{by \eqref{Ot2}}) \nonumber \\
&=& \sum_{k=1}^n \tau(p_k)\tau(t_k) \quad (\text{by independence}) \nonumber\\
&\geq & \min_{1 \leq k\leq n} \tau\left(t_k \right)\tau(p).
\end{eqnarray*}
\end{proof}


Now, we present a quantum version of L\'{e}vy--Skorohod inequality.

\begin{theorem}[Noncommutative L\'{e}vy--Skorohod inequality]
Let $x_1, x_2, \ldots, x_n$ be weakly fully independent self-adjoint random variables in $\mathfrak{M}$ with partial sums $s_k$ satisfying $s_ks_n = s_ns_k$ for $k=1, 2, \ldots, n$. For each $\lambda>0$ and for any $0 < \alpha <1$, if $M_{\lambda}:= \frac{1}{\min\limits_{1 \leq k\leq n} \tau\left(e_{-(1-\alpha)\lambda}^{\perp}\left(s_n - s_k \right) \right)}$ (with $1/0=\infty$), then there exists a projection $p$ such that
\begin{eqnarray}\label{LO1}
\max\limits_{1 \leq k \leq n} \frac{1}{2^{k-1}} \tau \left( e_{(\lambda, \infty)}\left(s_k \right) \right) ~ \leq ~\tau(p) ~ \leq ~ M_{\lambda}\, \tau \left( e_{(\alpha \lambda, \infty)}\left( s_n \right) \right)
\end{eqnarray}
\end{theorem}
\begin{proof}
Set $$r_k := e_{(-\infty, \lambda]}\left( s_k \right) \quad (1 \leq k \leq n)$$ as a sequence of spectral projections, and put
\begin{eqnarray*}
&&p_k := \bigwedge_{i=1}^{k-1} r_{i} \wedge r_k^{\perp};\\
\text{and} \quad
&&t_k := e_{[-(1-\alpha)\lambda ,\infty)}\left(s_n - s_k \right);
\end{eqnarray*}
for each $1 \leq k \leq n$. Note that $t_k r_k = r_k t_k$ and
\begin{eqnarray}\label{LO2}
t_k r_k^{\perp} \leq e_{(\alpha\lambda, \infty)}\left(s_n \right),
\end{eqnarray}
and hence
\begin{eqnarray*}
\sum_{k=1}^n \tau \left( t_k\, p_k \right) \leq \sum_{k=1}^n \tau \left( p_k e_{(\alpha\lambda, \infty)}\left( s_n \right) \right) \leq \tau \left( e_{(\alpha\lambda, \infty)}\left( s_n \right) \right).
\end{eqnarray*}
Putting $p := \sum_{k=1}^n p_k$, inequality \eqref{LO1} can be concluded from
\begin{eqnarray*}
\tau \left( e_{(\alpha\lambda, \infty)}\left( s_n \right) \right) &\geq & \sum_{k=1}^n \tau(p_k t_k) \quad \qquad (\text{by \eqref{LO2}}) \nonumber \\
&=& \sum_{k=1}^n \tau(p_k)\tau(t_k) \quad (\text{by independence}) \nonumber\\
&\geq & \min_{k=1}^n \tau\left(t_k \right)\tau(p).
\end{eqnarray*}
\end{proof}

Finally, we conclude the classical version of L\'{e}vy--Skorohod inequality.

\begin{corollary}
Let $X_1, X_2, \ldots, X_n$ be independent random variables in probability space $(\Omega, \mathfrak{F}, \mathbb{P})$ with partial sums $S_k$. Then, for any $\lambda > 0$ and $0 < \alpha < 1$
\begin{eqnarray*}
\mathbb{P}\left( \max\limits_{1\leq k \leq n}  S_k  > \lambda\right) \min\limits_{1 \leq k \leq n} \mathbb{P}\big( S_n - S_k \geq -(1 - \alpha )\lambda \big) \leq \mathbb{P} \left(S_n > \alpha\lambda \right)
\end{eqnarray*}
\end{corollary}

\section{Noncommutative strong and weak symmetrization inequalities}

Recall that a sequence $\left(x_k\right)_{k=1}^n$ in $\mathfrak{M}$ is \textit{tensor independent} \cite[Definition 2.5 and Remarks after that]{D} if
\begin{eqnarray*}
\tau\left(\prod_{i=1}^m\left(\prod_{k=1}^n a_{ki}\right) \right) = \prod_{k=1}^n\tau\left(\prod_{i=1}^m a_{ki}\right),
\end{eqnarray*}
whenever $a_{ki} \in W^*(x_k)$ ~ ($1 \leq i \leq m$; $1 \leq k \leq n$; $m \in \mathbb{N}$).

For any self-adjoint element in $\mathfrak{M}$, one may construct a tensor independent operator such that together with $x$ are identically distributed. To be more precise, consider the von Neumann algebra tensor product $\mathfrak{M}\otimes \mathfrak{M}^{\prime}$ in which $\mathfrak{M}^{\prime} = \mathfrak{M}$. Then any member of $\mathfrak{M}$ and $\mathfrak{M}^{\prime}$ can be regarded as elements of $\mathfrak{M}\otimes \mathfrak{M}^{\prime}$ equipped with the tensor product trace $\overline{\tau}$, which is uniquely determined by $\overline{\tau}(x\otimes x^{\prime})= \tau(x)\tau(x^{\prime})$ via the following maps, respectively:
\begin{eqnarray*}
x \in \mathfrak{M} \mapsto x\otimes 1 \quad \text{and} \quad x \in \mathfrak{M}^{\prime} \mapsto 1\otimes x.
\end{eqnarray*}
For any self-adjoint element $x \in \mathfrak{M}$, we put $\overline{x}=x\otimes1$ and $\overline{x}^{\prime}=1\otimes x$.

Clearly, $\overline{x}$ and $\overline{x}^{\prime}$ have identical moments in $ \mathfrak{M}\otimes \mathfrak{M}^{\prime}$ (i.e. $\overline{\tau}(\overline{x}^{~k})=\overline{\tau}(\overline{x}^{{\prime}^{~k}})$ for every $k\in\mathbb{N}$), so $\overline{x}$ and $\overline{x}^{\prime}$ are identically distributed by \cite[p. 203, Remark]{D}. Furthermore, $\overline{x}$ and $\overline{x}^{\prime}$ are tensor independent with respect to $\overline{\tau}$ by \cite[Theorem 3.1]{D}.\\

{\bf STANDING NOTATION.} To simplify our notation, by passing to $\mathfrak{M} \otimes \mathfrak{M}$ and considering an isometric copy of $\mathfrak{M}$ therein, we denote $\overline{x}$, $\overline{x}^{\prime}$, and $\overline{\tau}$ with $x$, $x^\prime$, and $\tau$, respectively; and set $\widehat{x} :=x-x^\prime$.\\


The spirit of the next proposition is related to symmetrization inequalities, which provide relations between random variables and symmetrized versions. As an application, symmetrization inequalities can be applied to relate moments of random variables to their symmetrized counterparts.

\begin{proposition}[Strong symmetrization inequality]\label{pr2}
Let $x_1, x_2, \ldots, x_n$ be self-adjoint random variables in $\mathfrak{M}$. For any $\lambda$, there exists a projection $p$ such that
\begin{eqnarray}\label{SSI1}
\tau(p) \leq 2\, \tau\left(\bigvee_{i=1}^n e_{\lambda}^{\perp}(\widehat{x_{i}})\right).
\end{eqnarray}
Moreover, the projection $p$ is nonzero provided that $e_{\lambda}^{\perp}(x_i - \med(x_i))$ is nonzero for some $1 \leq i \leq n$.
\end{proposition}
\begin{proof}
Put $z_k := x_k - \med(x_k)$ and set
\vspace{-5 mm}
\begin{eqnarray}
&&r_0:=1, \quad r_k := \bigwedge_{j=1}^k e_{\lambda}\left( z_k \right), \label{S1} \\
&&p_k := r_{k-1} \wedge e_{\lambda}^{\perp}\left( z_k \right),\label{S2}\\
\text{and} ~&& q_k := e_{\lambda}^{\perp}\left( \widehat{x_{k}} \right). \nonumber
\end{eqnarray}
Then $ (p_k)_k$ is a sequence of orthogonal projections. Infact, if $1 \leq k < j \leq n$, then $p_j$ is a sub-projection of $e_{\lambda}\left( z_k \right)$ and $p_k$ is a sub-projection of $e_{\lambda}^{\perp}\left( z_k \right)$, and hence $p_j p_k =0$ for all $1 \leq k < j \leq n$.\\
Putting $f_k := e_{(-\infty, 0]}\left( x_k^{\prime} - \med(x_k^{\prime}) \right)$, note that projections $e_{(\lambda, \infty)}\left( z_k \right), f_k$, and $q_k$ commute with each other, because $x$ and $x^{\prime}$ as well as their spectral projections commute. Considernig the abelian von Neumann algebra $W^*\left( x_k, x_k^{\prime} \right)$ and noting that $\med(x_k) = \med(x_k^{\prime})$, we conclude that
\begin{eqnarray}\label{st2}
e_{(\lambda, \infty)}\left( z_k \right)\, f_k = e_{(\lambda, \infty)}\left( z_k \right) \wedge f_k \leq q_k.
\end{eqnarray}
Now, from \eqref{st2} we get
\begin{eqnarray*}
p_k\, f_k\, p_k = p_k\, e_{(\lambda, \infty)}\left( z_k \right) f_k\, p_k \leq p_k\, q_k\, p_k \leq p_k\, \left( \vee_{i=1}^n q_i \right)\, p_k
\end{eqnarray*}
Therefore,
\begin{eqnarray}\label{st3}
\sum_{k=1}^n \tau \left(p_k\, f_k \right) \leq \sum_{k=1}^n \tau \left( p_k\, \left( \vee_{i=1}^n q_i \right) \right).
\end{eqnarray}
As $p_k$'s are orthogonal, setting $p:= \sum_{k=1}^n p_k$, we obtain
\begin{eqnarray*}
\tau \left(\vee_{i=1}^n q_i\right) &\geq & \sum_{k=1}^n \tau \left( p_k\, \left( \bigvee\limits_{i=1}^n q_i \right) \right) \\
&=& \sum_{k=1}^n \tau (p_k\, f_k) \qquad\qquad\quad(\text{by \eqref{st3}}) \\
&=& \sum_{k=1}^n \tau (p_k) \tau (f_k) \qquad\qquad(\text{by the tensor independence})\\
&\geq & \frac{1}{2} \sum_{k=1}^n \tau (p_k)\\
& = & \frac{1}{2} \tau (p),
\end{eqnarray*}
in which the last inequality follows from the definition of median; hence we have proved \eqref{SSI1}.

To investigate the last assertion, suppose that $p=0$; so $p_k = 0$ for all $1 \leq k\leq n$. However, it is easy to see from \eqref{S1}, \eqref{S2}, and induction that  $e_{(\lambda, \infty)}\left( z_k \right) = 0$ for all $1 \leq k \leq n$, which is contrary to our assumptions.
\end{proof}
\begin{corollary}
Let $X_k, k=1, 2, \ldots, n$, be classical random variables, and let $\widehat{X_{k}} = X_k - X_k^{\prime}$, where $X_k^{\prime}$ is an independent copy of $X_k$ for each $k$. Then for each $\lambda \in \mathbb{R}$,
\begin{eqnarray*}
\mathbb{P}\left(\max_{1 \leq k \leq n} (X_k - \med(X_k)) \geq \lambda\right) \leq 2 \mathbb{P}\left(\max_{1\leq k \leq n} (\widehat{X_{k}}) \geq \lambda\right).
\end{eqnarray*}
\end{corollary}
\begin{proof}
It is easy to check that the projections $p$ and $\bigvee_{k=1}^n q_k$ in Proposition \ref{pr2} correspond to the characteristic functions of the subsets $\{\max_{k=1}^n (X_k - \med(X_k)) > \lambda\}$ and $\{\max_{k=1}^n (\widehat{X_{k}}) > \lambda\}$, respectively.
\end{proof}
\begin{corollary}[Weak symmetrization inequality]
Let $x \in \mathfrak{M}$ be a self-adjoint operator. If $m = \med(x)$, then for any $\lambda \in \mathbb{R}$  and $\alpha \in \mathbb{R}$,
\begin{eqnarray*}
&& \tau\left( e_{\lambda}^{\perp} (x - m))\right)~ \leq ~ 2\, \tau \left(e_{\lambda}^{\perp}(\widehat{x})\right)\\
&& \tau \left( e_{\lambda}^{\perp}(|x - m|)\right)~ \leq ~ 2\, \tau \left(e_{\lambda}^{\perp}(|\widehat{x}|)\right) ~ \leq ~ 4\, \tau \left( e_{\frac{\lambda}{2}}^{\perp}(|x - \alpha |)\right).
\end{eqnarray*}
In particular,
\begin{eqnarray*}
\tau \left( e_{\lambda}^{\perp}(|x - m|)\right)~ \leq ~ 2\, \tau \left(e_{\lambda}^{\perp}(|\widehat{x}|)\right) ~ \leq ~ 4\, \tau \left( e_{\frac{\lambda}{2}}^{\perp}(|x - m |)\right).
\end{eqnarray*}
\end{corollary}
\begin{proof}
Applying the strong symmetrization inequality with $n=1$ and $x_1 = x$ to get
\begin{eqnarray}\label{We1}
\tau \left( e_{\lambda}^{\perp}(x - m)\right) \leq 2\,  \tau \left(e_{\lambda}^{\perp}(\widehat{x})\right).
\end{eqnarray}
Applying the strong symmetrization inequality with $n=1$ and $x_1 = -x$ gives
\begin{eqnarray}\label{We2}
\tau \left( e_{\lambda}^{\perp}(-x + m)\right) \leq 2\,  \tau \left(e_{\lambda}^{\perp}(-\widehat{x})\right).
\end{eqnarray}
Note that for any self-adjoint operator $y$ it holds that $e_{\lambda}^{\perp}(|y|) = e_{\lambda}^{\perp}(y) + e_{\lambda}^{\perp}(-y)$. Summing \eqref{We1} and \eqref{We2} we get $$\tau \left( e_{\lambda}^{\perp}(|x - m|)\right) \leq 2 \tau \left(e_{\lambda}^{\perp}(|\widehat{x}|)\right).$$
Finally, the right hand side of the second assertion can be obtained from
\begin{eqnarray*}
\tau \left(e_{\lambda}^{\perp}(|\widehat{x}|)\right) &=& \tau \left(e_{\lambda}^{\perp}\left( \big|\, x- \alpha - (x^{\prime} - \alpha)\, \big|\right)\right)\\
&\leq & \tau \left(e_{\frac{\lambda}{2}}^{\perp}(|x - \alpha |)\right) + \tau \left(e_{\frac{\lambda}{2}}^{\perp}(|x^{\prime} - \alpha |)\right)\\
&=& 2\, \tau \left(e_{\frac{\lambda}{2}}^{\perp}(|x - \alpha |)\right).
\end{eqnarray*}
Note that $x$ and $x^{\prime}$ are commutating, and the above inequality easily follows from Borel functional calculus. The last equality comes from the fact that if $x, y \in \mathfrak{M}_{sa}$ are identically distributed, then $f(x)$ and $f(y)$ are identically distributed for every continuous function $f$ on the union of the spectra of $x, y$. We refer the reader to \cite[Page 203, Remark]{D} for more details. 
\end{proof}
The following result can be concluded from the weak symmetrization inequality. Recall that a sequence $(x_n)$ in $\mathfrak{M}$ \textit{converges to $x$ in measure} whenever $\tau\left(e_{\epsilon}^{\perp}(|x_n - x|) \right)$ converges to zero for any $\epsilon >0$. (see \cite{Y} for more details). 

\begin{corollary}
For a sequence $(x_n)$ in $\mathfrak{M}^{sa}$ and a sequence $(\alpha_n)$ of real numbers, if $x_n - \alpha_n \longrightarrow 0$ in measure, then $\widehat{x_n} \longrightarrow 0$ in measure and $\alpha_n - \med(x_n) \longrightarrow 0$. 

In particular, if $x_n \longrightarrow 0$ in measure, then $\med(x_n) \longrightarrow 0$.
\end{corollary}

\begin{corollary}
Let $x \in \mathfrak{M}^{sa}$. Then, for any $\alpha \in \mathbb{R}$ and $p \geq 1$, it holds that
\begin{eqnarray*}
\frac{1}{2} \left\| x - \med(x) \right\|_p^p ~\leq ~ \left\| \widehat{x} \right\|_p^p ~\leq ~ 2K_p \| x - \alpha \|_p^p
\end{eqnarray*}
for some constant $K_p$.
\end{corollary}
\begin{proof}
\begin{eqnarray*}
\left\| x - x^{\prime} \right\|_p^p &=& \left\| (x -\alpha) - (x^{\prime} - \alpha) \right\|_p^p\\
&\leq& K_p (\left\| x -\alpha \right\|_p^p + \left\| x^{\prime} - \alpha \right\|_p^p)\\
&& \qquad\quad\quad\qquad (\text{by the noncommutative Clarkson inequality \cite{FK}})\\
&=& 2K_p \left\| x -\alpha \right\|_p^p = 2K_p \| x -\alpha \|_p^p.
\end{eqnarray*}
Thanks to the weak symmetrization inequality, the left hand side follows via
\begin{eqnarray*}
\| x - \med(x) \|_p^p &=& \left\| x - \med(x) \right\|_p^p = \int_{0}^{\infty} pt^{p-1}\, \overline{\tau} \left( e_{[t, \infty)}(\left| x - \med(x)\right|)\right) dt\\
&\leq & 2\, \int_{0}^{\infty} pt^{p-1}\, \tau \left(e_{[t, \infty)}(|\widehat{x}|)\right)dt = 2\, \left\| \widehat{x} \right\|_p^p.
\end{eqnarray*}
\end{proof}

\textbf{Acknowledgement.} The authors would like to sincerely thank the referee for carefully reading the paper and for giving some helpful comments improving it. The third author was supported by a postdoctoral grant from Iran's National Elites Foundation (INEF) under the supervision of the third author.


\begin{thebibliography}{99}

\bibitem{BEK} T. N. Bekjan and Z. Chen, \textit{Interpolation and $\Phi$-moment inequalities of noncommutative martingales}, Probab. Theory and Related Fields \textbf{152} (2012), no. 1-2, 179--206.

\bibitem{BS} M. B\.ozejko and R. Speicher. \textit{$\psi$-independent and symmetrized white noises}, Quantum probability \& related topics, 219--236, QP-PQ, VI, World Sci. Publ., River Edge, NJ, 1991.

\bibitem{BK} L. G. Brown and H. Kosaki, \textit{Jensen's inequality in semi-finite von Neumann algebras}, J. Operator Theory  \textbf{23} (1990), no. 1, 3--19.

\bibitem{CF} V. Crismale and F. Fidaleo. \textit{Exchangeable stochastic processes and symmetric states in quantum probability}, Ann. Mat. Pura Appl. (4) \textbf{194} (2015), no. 4, 969--993.

\bibitem{D} S. Dirksen, \textit{Noncommutative stochastic integration through decoupling}, J. Math. Anal. Appl. \textbf{370} (2010), no. 1, 200--223.

\bibitem{FK} T. Fack and H. Kosaki, \textit{Generalized $s$-numbers of $\tau$-measurable operators}, Pacific J. Math. \text{123} (1986), no. 2, 269--300.

\bibitem{Fr1} U. Franz, \textit{Monotone independence is associative}, Infin. Dimens. Anal. Quantum Probab. Relat. Top. \textbf{4} (2001), no. 3, 401--407.

\bibitem{Fr3} U. Franz, \textit{Multiplicative monotone convolutions}, In: Quantum probability. Vol. 73. Banach Center Publ. Polish Acad. Sci., Warsaw, 2006, pp. 153--166. 

\bibitem{Jiao} Y. Jiao, A. Osekowski, and L. Wu, \textit{Noncommutative good-$\lambda$ inequality}, arXiv:1802.07057v1.

\bibitem{SUK} Y. Jiao, F. Sukochev, and D. Zanin, \textit{Sums of independent and freely independent identically distributed random variables}, Studia Math. \textbf{251} (2020), no. 3, 289--315.

\bibitem{JX2} M. Junge and Q. Xu, \textit{Noncommutative Burkholder/Rosenthal inequalities II: Applications}, Israel J. Math. \textbf{167} (2008), 227--282.

\bibitem{L} P. L\'{e}vy, \textit{Th\'eorie de l'addition des Variables al\'eatoires}, (French), Paris: Gauthier-Villars, 1954.

\bibitem{LR} D. Li, Y. Qi, and A. Rosalsky, \textit{A characterization of a new type of strong law of large numbers}, Trans. Amer. Math. Soc. \textbf{368} (2016), no. 1, 539--561.

\bibitem{Liu} W. Liu, \textit{A noncommutative definetti theorem for Boolean independence}, J. Funct. Anal. \textbf{269} (2015), 1950--1994.

\bibitem{Lu} A. Luczak, \textit{Laws of large numbers in von Neumann algebras and related results}, Studia Math. \textbf{81} (1985), no. 3, 231--243.

\bibitem{Mur} N. Muraki, \textit{Monotonic independence, monotonic central limit theorem and monotonic law of large numbers}, Inf. Dim. Anal. Quant. Probab. Rel. Topics, \textbf{4} (2001), 39--58.

\bibitem{N} E. Nelson, \textit{Notes on non-commutative integration}, J. Funct. Anal. \textbf{15} (1974), 103--116.

\bibitem{PIN} I. Pinelis, \textit{Exact Rosenthal-type bounds}, Ann. Probab. \textbf{43} (2015), no. 5, 2511--2544.

\bibitem{R} N. Randrianantoanina, \textit{Conditioned square functions for noncommutative martingales}, Ann. Probab. \textbf{35} (2007), no. 3, 1039--1070.

\bibitem{SAK} A. I. Sakhanenko, \textit{L\'{e}vy-Kolmogorov inequalities for random variables with values in a Banach space}, (Russian) Teor. Veroyatnost. i Primenen. \textbf{29} (1984), no. 4, 793--799.

\bibitem{SZE} Z. S. Szewczak, \textit{On the maximal L\'{e}vy--Ottaviani inequality for sums of independent and dependent random vectors}, Bull. Pol. Acad. Sci. Math. \textbf{61} (2013), no. 2, 155--160.

\bibitem{TMS1} A. Talebi, M. S. Moslehian, and Gh. Sadeghi, \textit{Etemadi and Kolmogorov inequalities in noncommutative probability spaces}, Michigan. Math. J.  \textbf{68} (2019), no. 1, 57--69.

\bibitem{TMS2} A. Talebi, M. S. Moslehian, and Gh. Sadeghi, \textit{Noncommutative Blackwell--Ross martingale inequality}, Infin. Dimens. Anal. Quantum Probab. Relat. Top. \textbf{21} (2018), no. 1, 1850005, 9 pp.

\bibitem{VDN} D. V. Voiculescu, K. J. Dykema, and A. Nica, \textit{Free random variables}, volume 1 of CRM Monograph Series, American Mathematical Society, Providence, RI, 1992.

\bibitem{Y} F. J. Yeadon, \textit{Non-commutative $L_{p}$-spaces}, Math. Proc. Cambridge Philos. Soc. \textbf{77} (1975), 91--102.

\end{thebibliography}
\end{document}